\documentclass[11pt,reqno]{amsart}

\usepackage{amsthm}
\usepackage{amsmath}
\usepackage{amsfonts}
\usepackage{amssymb}
\usepackage{bbm}
\usepackage{enumerate}
\usepackage{caption}
\usepackage{subcaption}
\usepackage{multirow}
\usepackage{mathrsfs}  
\usepackage{pgf,tikz}
\usetikzlibrary{matrix,arrows}
\usepackage{graphicx}
\usepackage[margin=1in]{geometry} 

\newtheorem{thm}{Theorem}

\newtheorem{cor}[thm]{Corollary}
\newtheorem{prop}[thm]{Proposition}
\newtheorem{lem}[thm]{Lemma}

\theoremstyle{definition}

\newcommand{\R}{\mathbbm{R}}
\newcommand{\N}{\mathbbm{N}}
\newcommand{\Z}{\mathbbm{Z}}

\newcommand{\C}{\mathbbm{C}}
\newcommand{\coloring}{\mathcal{C}}
\newcommand{\cmap}{\varphi_{\alpha}}
\newcommand{\tiling}{\mathscr{T}}
\newcommand{\stiling}{\mathscr{S}_{\alpha}}
\newcommand{\qalpha}{Q_{\alpha}}

\DeclareMathOperator{\sg}{S}
\DeclareMathOperator{\rg}{R}
\DeclareMathOperator{\tg}{T}
\DeclareMathOperator{\Orth}{O}
\DeclareMathOperator{\interior}{int}

\begin{document}

	\author{Imogene F.~Evidente}
	\address[I.F.~Evidente]{Institute of Mathematics, University of the Philippines Diliman, 1101 Quezon City, Philippines}
	\email[Corresponding Author]{ifevidente@up.edu.ph}
	
	\author{Rene P.~Felix}
	\address[R.P.~Felix]{Institute of Mathematics, University of the Philippines Diliman, 1101 Quezon City, Philippines}
	\email{rene@math.upd.edu.ph}
	
	\author{Manuel Joseph C.~Loquias}
	\address[M.J.C.~Loquias]{Institute of Mathematics, University of the Philippines Diliman, 1101 Quezon City, Philippines, and Chair of Mathematics and Statistics, University of
	Leoben, Franz-Josef-Strasse 18, A-8700 Leoben, Austria}
	\email{mjcloquias@math.upd.edu.ph}

	\title[Symmetries and Color Symmetries of a Family of Tilings with a Singular Point]{Symmetries and Color Symmetries of\\[3pt]a Family of Tilings with a Singular Point}

	\begin{abstract}
		We obtain tilings with a singular point by applying conformal maps on regular tilings of the Euclidean plane, and determine its symmetries. 
		The resulting tilings are then symmetrically colored by applying the same conformal maps on colorings of regular tilings arising from sublattice colorings of the centers 
		of its tiles.  In addition, we determine conditions so that the coloring of a tiling with singularity that is obtained in this manner is perfect.
	\end{abstract}	
	


	\date{\today}

	\maketitle
	
	\section{Introduction and Outline}   

		A singularity of a tiling is a point $P$ for which any circular disk centered at $P$ intersects an infinite number of tiles. 
		Hence, tilings with singular points are not locally finite and are classified among tilings that are not ``well-behaved'' \cite{Grun}.
		Geometric and topological properties of tilings with singularities have been investigated in \cite{Bree,Niel,Sush}.
		
		Symmetric colorings of tilings with a singular point were obtained in \cite{Luck} by distorting certain colorings of regular Euclidean tilings.
		However, it was found that not all colorings of regular tilings could be transformed into colorings of tilings with a singular point, because some colorings were incompatible 
		with the tiling's rotational symmetry.
		Furthermore, it was surmised that there is a maximum number of colors for such colorings of a tiling with singularity.
		In Figure \ref{fig:(4^4)Example1a}, we have recreated a coloring of a tiling with a singularity from \cite[Figure 6]{Luck}. 
		Figure \ref{fig:(4^4)Example1b} shows the coloring of a regular tiling by squares from which Figure \ref{fig:(4^4)Example1a} was obtained. 
		Other colorings of the square tiling also compatible with the rotational symmetry of the tiling with singularity were found to have $1$, $2$, $5$, $10$, $25$ and $50$ colors only.
	
		\begin{figure}[ht]
			\begin{subfigure}[c]{0.4\textwidth}
				\includegraphics[height=4.5cm]{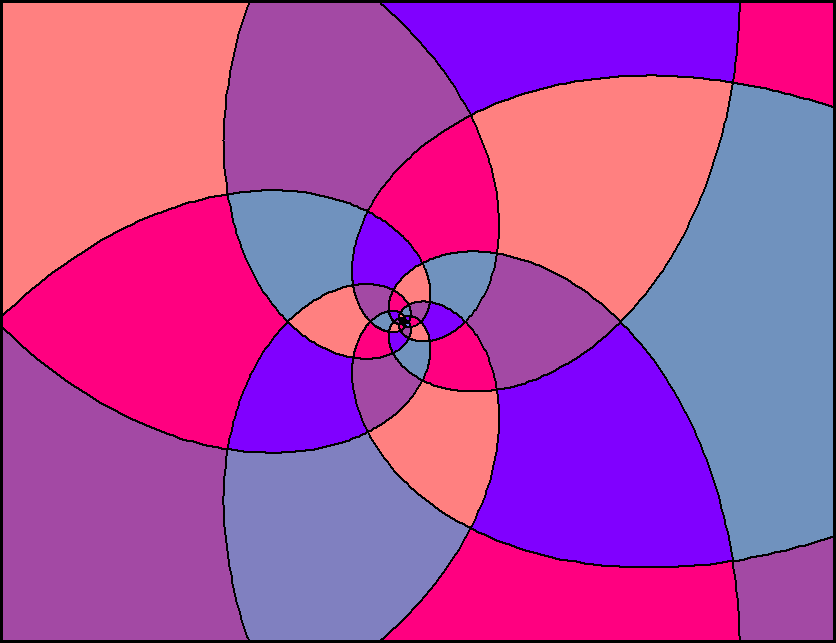}
				\caption{}
				\label{fig:(4^4)Example1a}
			\end{subfigure}
			\begin{subfigure}[c]{0.3\textwidth}
				\includegraphics[height=4.5cm]{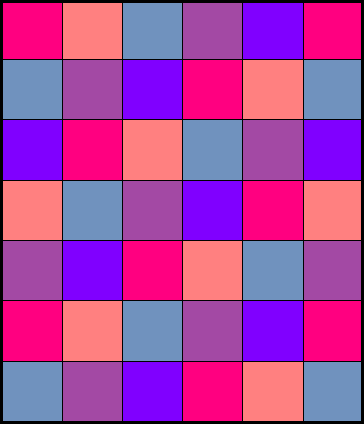}
				\caption{}
				\label{fig:(4^4)Example1b}
			\end{subfigure}
			\caption{Coloring of a tiling with a singularity and corresponding coloring of the $(4^4)$ tiling}
		\end{figure}
	
		In this contribution, we provide a mathematical basis for the observations made in \cite{Luck}. 
		We begin by introducing in Section \ref{sec:Preliminaries} the framework used in this paper, together with some known results.
		Section \ref{sec:ConformalMapping} discusses the conformal maps that are applied on regular Euclidean tilings which yield tilings with a singularity. 
		In Section \ref{sec:SymmetryGroup}, a correspondence between the the symmetry groups of the regular tilings and their images under the conformal maps is established.
		This allows us to determine the symmetry group of the tiling with singularity, which turns out to be isomorphic to a finite cyclic or finite dihedral group.  
		Section \ref{sec:ColorSymmetry} focuses on the symmetry of colorings of tilings with a singularity.
		We use a method similar to that in \cite{Luck} to obtain colorings of tilings with a singularity: conformal maps are applied on sublattice colorings of regular Euclidean tilings.
		However, applying a conformal map on a coloring of a regular tiling does not always yield a coloring of the corresponding tiling with singularity.
		Consequently, a compatibility condition between a sublattice coloring of a regular tiling and a conformal map is established.
		Furthermore, for colorings compatible with a conformal map, we identify a necessary and sufficient condition so that a symmetry of the uncolored tiling with singularity is
		a color symmetry of the resulting coloring of the tiling.		
		Lastly, we determine the maximum number of colors of a perfect coloring of the tiling with a singularity obtained in this manner.
		Several examples are presented in Section \ref{sec:Examples}.
		
	\section{Preliminaries}\label{sec:Preliminaries}

		A planar \emph{tile} is any set $T$ in the Euclidean plane that is the closure of its interior.
		Here, tiles are always bounded, and consequently, compact.
		A \emph{tiling} of $\R^2$ is a set of tiles $\mathscr{T}=\{ T_i\} _{i\in\N}$, such that ${\cup}_{i\in\N}T_i=\R^2$ and $\interior(T_i)\cap\interior(T_j)=\varnothing$ for all 
		$i\neq j$.
		That is, $\mathscr{T}$ is both a packing and a covering of $\R^2$.
		
		A tiling $\mathscr{T}$ of $\R^2$ is said to be \emph{locally finite} if for any given point $x\in\R^2$, every circular disk centered at $x$ intersects only a finite number of 
		tiles of $\mathscr{T}$  \cite{Grun,BaakeGrimm}.
		Hence, if $\mathscr{T}$ is not locally finite, then there exists at least one point $x_0\in\R^2$ such that for all $r>0$, $B(x_0,r)\cap T\neq\varnothing$ for an infinite 
		number of $T\in\mathscr{T}$.
		Such a point is called a \emph{singular point} of the tiling.
		
		We are primarily concerned with symmetries of tilings with a singularity at the origin that are obtained by applying conformal maps on the three types of regular tilings of 
		$\R^2$: the $(4^4)$ tiling (regular tiling by squares), the $(6^3)$ tiling (regular tiling by hexagons), and the $(3^6)$ tiling (regular tiling by triangles). Furthermore, we
		determine the color symmetries of colorings of tilings with a singularity obtained from certain colorings of regular tilings. To this end, the following definitions taken from 
		\cite{Grun} are needed.
		
		Let $X$ be a set of points or a set of tiles in $\R^2$, with symmetry group $\sg(X)$. We refer to the subgroup of $\sg(X)$ containing all direct symmetries of $X$ as the 
		\emph{rotation group} of $X$, and denote it by $\rg(X)$. Meanwhile, the subgroup of $\sg(X)$ containing all translation symmetries of $X$ is called the \emph{translation group} 
		of $X$, and is written $\tg(X)$.
		
		A \emph{coloring} of $X$ is a surjective map $\coloring_X:X\rightarrow K$, where $K$ is a finite set of $m$ colors.
		Equivalently, a coloring $\coloring_X$ of $X$ may be viewed as a partition $\{ X_i\} _{i=1}^m$ of $X$.
		A \emph{color symmetry} of $\coloring_X$ is an element of $\sg(X)$ that permutes the colors of $\coloring_X$.
		The set $\mathcal{H}_{\coloring_X}$ of all color symmetries of $\coloring_X$, that is
		\[\mathcal{H}_{\coloring_X}=\{h\in\sg(X)\mid \exists\text{ permutation }\sigma_h\text{ on }K\text{ such that }\forall\ x\in X,\ \coloring_X(h(x))=\sigma_h(\coloring_X(h))\},\]
		forms a subgroup of $\sg(X)$.
		We call this the \emph{color symmetry group} of the coloring \emph{$\coloring_X$}.
		If $\mathcal{H}_{\coloring_X}=\sg(X)$, then $\coloring_X$ is said to be a \emph{perfect coloring of} $X$.
		If  $\mathcal{H}_{\coloring_X}$ contains $\rg(X)$, then $\coloring_X$ is called a \emph{chirally perfect coloring of} $X$ \cite{Rigb}.
		To distinguish a perfect coloring from a chirally perfect coloring, we sometimes refer to a perfect coloring as a \emph{fully perfect coloring}.
		
		Of particular interest are \emph{sublattice colorings} of square and hexagonal lattices \cite{Mood,Pena} . 
		A sublattice $\Gamma$ of a lattice $\Gamma_0$ is a subgroup of finite index $m$ in $\Gamma_0$. 
		We assign a unique color to each coset of $\Gamma$ to obtain a coloring of $\Gamma_0$. 
		That is, the coloring of $\Gamma_0$ induced by $\Gamma$ is given by $\coloring_{\Gamma_0} = \{X_i\}_{i=1}^{m}$, where $X_i$, $i\in\{1,\ldots,m\}$,  is a coset of $\Gamma$. 
		Since the number of colors in the coloring of $\Gamma_0$ induced by $\Gamma$ is the index of $\Gamma$ in $\Gamma_0$,
		we sometimes refer to the number of colors as the \emph{color index} of the coloring. 
		Furthermore, the translation symmetries of $\Gamma_0$ are always contained in $\mathcal{H}_{\coloring_{\Gamma_0}}$ because they fix every coset of $\Gamma$ \cite[Theorem 1]{Pena}.
		
		Finally, we describe the setting that we will use throughout the paper.
		
		We identify the Euclidean plane $\R^2$ with the complex plane $\C$. 
		If $z\in\C$, we denote the modulus of $z$ by $|z|$, and its conjugate by $\bar{z}$. 
		We use $i$ and $\omega$ to represent $\sqrt{-1}$ and $e^{2\pi i/3}$, respectively. 
		Furthermore, we associate the square lattice with the ring $\Z[i]=\{a+bi\mid a,b\in\Z\}$ of Gaussian integers, while the hexagonal lattice is identified with the ring 
		$\Z[\omega]=\{a+b\omega\mid a,b\in\Z\}$ of Eisenstein integers. It is known that if $\xi=i$ or $\xi=\omega$, then $\Z[\xi]$ is a principal ideal domain.  Also, given
		$\beta\in\Z[\xi]$, the ideal $(\beta)$ generated by $\beta$ is of index $|\beta|^{2}$ in $\Z[\xi]$.
		
		Let $\tiling$ be a regular tiling, $P$ be the set of centers of tiles of $\tiling$, and $\Lambda$ be the orbit of the origin under the action of $\tg(\tiling)$.  We superimpose the
		complex plane over $\tiling$ in the following manner.
		\begin{description}
			\item[$\tiling$ is the $(4^4)$ tiling] The set $P$ is a square lattice. We superimpose $\C$ over $\tiling$ so that $P$ coincides with the points of $\Z[i]$. 
			Note that $\Lambda=\Z[i]$.
		
			\item[$\tiling$ is the $(6^3)$ tiling] The set $P$ forms a hexagonal lattice. We view $\tiling$ in $\C$ in such a way that $P$ coincides with $\Z[\omega]$, from which
			follows $\Lambda=\Z[\omega]$.
		
			\item[$\tiling$ is the $(3^6)$ tiling] The set $P$ is not a lattice, but nonetheless forms a crystallographic point packing \cite{ConwaySloane,BaakeGrimm}.  Superimpose $\C$ 
			over $\tiling$ so that the set $P_0$ of vertices of $\tiling$ coincides with the points of the ideal $(2+\omega)$ of $\Z[\omega]$. The two other cosets of $P_0$ in $\Z[\omega]$ 
			are $P_1=-1+(2+\omega)$ and $P_2=1+(2+\omega)$. Note that $P=P_1\cup P_2$, while $\Lambda =P_0$. This is illustrated in Figure \ref{(p0p1p2}.
		\end{description}
		
		\begin{figure}[ht]
			\centering
			\begin{tikzpicture}[scale=1.2]
				\draw[line width=1.5] (-2.58,2.98)--(2.58,2.98)--(2.58,-0.99)--(-2.58,-0.99)--(-2.58,2.98);
				\foreach \k in {-2,-1,0,1,2}
				\draw[line width=0.85] (.86*\k,-1)--(0.86*\k,3);
				\draw[line width=0.75] (-2.58,2.5)--(2.58,-.5);
				\draw[line width=0.75] (-2.58,1.5)--(1.72,-1);
				\draw[line width=0.75] (-2.58,0.5)--(0,-1);
				\draw[line width=0.75] (-2.58,-0.5)--(-1.72,-1);
				\draw[line width=0.75] (-1.72,3)--(2.59,0.5);
				\draw[line width=0.75] (2.58,1.5)--(0,3);
				\draw[line width=0.75] (1.72,3)--(2.58,2.5);
				\draw[line width=0.75] (-1.72,3)--(-2.58,2.5);
				\draw[line width=0.75] (0,3)--(-2.58,1.5);
				\draw[line width=0.75] (1.72,3)--(-2.58,0.5);
				\draw[line width=0.75] (2.58,2.5)--(-2.58,-.5);
				\draw[line width=0.75] (2.58,1.5)--(-1.72,-1);
				\draw[line width=0.75] (2.58,0.5)--(0,-1);
				\draw[line width=0.75] (2.58,-0.5)--(1.72,-1);
				\foreach \l in {0,1,2}
					\fill[cm={1,0,0,1,(0,\l)}] (0,0) circle (1.9pt);
				\foreach \l in {-1,0,1,2}
					\fill[cm={1,0,0,1,(0,\l)}] (-0.86,0.5) circle (1.9pt);
				\foreach \l in {0,1,2}
					\fill[cm={1,0,0,1,(0,\l)}] (-1.72,0) circle (1.9pt);
				\foreach \l in {-1,0,1,2}
					\fill[cm={1,0,0,1,(0,\l)}] (0.86,0.5) circle (1.9pt);
				\foreach \l in {0,1,2}
					\fill[cm={1,0,0,1,(0,\l)}] (1.72,0) circle (1.9pt);
				\foreach \l in {0,1,2}
					\fill[color=red,cm={1,0,0,1,(0.573,\l)}] (0,0) circle (1.9pt);
				\foreach \l in {-1,0,1,2}
					\fill[color=red,cm={1,0,0,1,(0.573,\l)}] (-0.86,0.5) circle (1.9pt);
				\foreach \l in {0,1,2}
					\fill[color=red,cm={1,0,0,1,(0.573,\l)}] (-1.72,0) circle (1.9pt);
				\foreach \l in {-1,0,1,2}
					\fill[color=red,cm={1,0,0,1,(0.573,\l)}] (0.86,0.5) circle (1.9pt);
				\foreach \l in {0,1,2}
					\fill[color=red,cm={1,0,0,1,(0.573,\l)}] (1.72,0) circle (1.9pt);
				\foreach \l in {-1,0,1,2}
					\fill[color=red,cm={1,0,0,1,(0.573,\l)}] (-2.58,0.5) circle (1.9pt);
				\foreach \l in {0,1,2}
					\fill[color=blue,cm={1,0,0,1,(-0.573,\l)}] (0,0) circle (1.9pt);
				\foreach \l in {-1,0,1,2}
					\fill[color=blue,cm={1,0,0,1,(-0.573,\l)}] (-0.86,0.5) circle (1.9pt);
				\foreach \l in {0,1,2}
					\fill[color=blue,cm={1,0,0,1,(-0.573,\l)}] (-1.72,0) circle (1.9pt);
				\foreach \l in {-1,0,1,2}
					\fill[color=blue,cm={1,0,0,1,(-0.573,\l)}] (0.86,0.5) circle (1.9pt);
				\foreach \l in {0,1,2}
					\fill[color=blue,cm={1,0,0,1,(-0.573,\l)}] (1.72,0) circle (1.9pt);
				\foreach \l in {-1,0,1,2}
					\fill[color=blue,cm={1,0,0,1,(-0.573,\l)}] (2.58,0.5) circle (1.9pt);
			\end{tikzpicture}
			\caption{\textbf{$(3^6)$ Tiling}: The black, blue and red points represent $P_0=(2+\omega)$, $P_1=-1+(2+\omega)$ and $P_2=1+(2+\omega)$, respectively}\label{(p0p1p2}
		\end{figure}
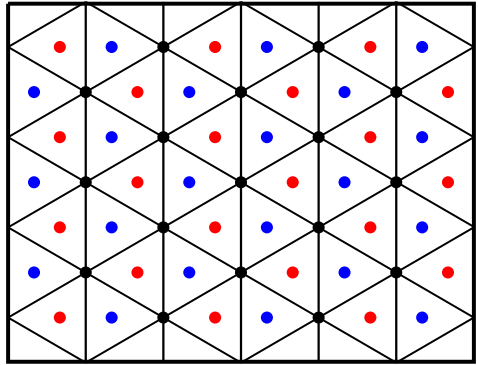
		
		The symmetry group of $\tiling$ is $\sg(\tiling)=\langle h_1,h_2,h_3,h_4 \rangle$, where $h_1$ is a rotation, $h_2$ is a reflection, and $h_3$ and $h_4$ are translations. 
		Table \ref{tab:SymmetryGroupT} shows these symmetries as complex functions for the three regular tilings.				
		\begin{table}[ht]
			\begin{tabular}{|c|c|c|c|}
				\hline
				Tiling & $(4^4)$ & $(6^3)$ & $(3^6)$ \\
				\hline
				$h_1(z)=$ & $iz$ &  $(1+\omega)z$ & $(1+\omega)z$   \\
				\hline
				$h_2(z)=$ & $\bar{z}$ & $\bar{z}$ & $\bar{z}$ \\
				\hline
				$h_3(z)=$ & $z+1$ & $z+1$ & $z+(2+\omega)$ \\
				\hline
				$h_4(z)=$ & $z+i$ & $z+\omega$ & $z+(1-\omega)$ \\
				\hline
				$\sg(\tiling)$ of type & $p4m$ or $\ast 442$ & $p6m$ or $\ast 632$ & $p6m$ or $\ast 632$\\
				\hline
				\end{tabular}
			\caption{Symmetries of Regular Tilings}\label{tab:SymmetryGroupT}
		\end{table}
		
		We need the following result from \cite{Buga} that provides necessary and sufficient conditions for a sublattice coloring of $\Z[i]$ or $\Z[\omega]$ to be chirally or fully
		perfect.  
		
		\begin{thm}\label{thm:PerfectT}
			Let $\xi\in\{i,\omega\}$, and $\coloring_{\Z[\xi]}$ be a coloring induced by the sublattice $\Gamma$. Then the coloring $\coloring_{\Z[\xi]}$ is chirally perfect if and only if 
			$\Gamma$ is an ideal of $\Z[\xi]$. In addition, if $\beta\in\Z[\xi]$, then the coloring $\coloring_{\Z[\xi]}$ induced by $(\beta)$ is perfect if and only if 
			$\beta=\pm\xi^j\bar{\beta}$.
		\end{thm}
		
		Theorem \ref{thm:PerfectT} states that a sublattice coloring of $\Z[i]$ or $\Z[\omega]$ is chirally perfect if and only if the sublattice that induces the coloring is an ideal of 
		the respective lattice. We refer to such colorings as \emph{ideal colorings} of the lattice. Furthermore, this ideal coloring is perfect if its generator is \emph{balanced}, 
		that is, the generator and its conjugate are associates \cite{Wash}.

	\section{Family of Tilings with a Singularity via Conformal Mappings}\label{sec:ConformalMapping}
		We consider the conformal map $\cmap:\C\rightarrow\C^{\ast}:=\C\setminus\{0\}$ defined by 
		\begin{equation}\label{eq:ConformalMap}
			\cmap(z)=\exp\left(\frac{2\pi i z}{\alpha}\right)=\exp\left(\frac{2\pi i}{{|\alpha|}^2}\bar{\alpha}z\right)
		\end{equation}
		where $\alpha\in\C^{\ast}$.  The following lemma establishes which points of $\C$ have the same image under $\cmap$.
		
		\begin{lem}\label{lem:SameImage}
			Let $\cmap(z)=\exp(2\pi i z/\alpha)$, where $\alpha\in\C^{\ast}$. 
			Given $z_1,z_2\in\C$, then $\cmap(z_1)=\cmap(z_2)$ if and only if $z_2=z_1+k\alpha$ for some $k\in\Z$.
		\end{lem}
		
		We now apply $\cmap$ on a regular tiling $\tiling$ to obtain a tiling with a singularity at the origin. For this, we need to find conditions on $\alpha$ that guarantee that 
		$\stiling:=\{\cmap(T)\mid T\in\tiling\}$ is a tiling of $\C^{\ast}$.
		
		\begin{thm}\label{alphaconditions}
			Let $\tiling$ be a regular tiling, $\Lambda$ be the orbit of the origin under the action of $\tg(\tiling)$ and $\cmap(z)=\exp(2\pi i z/\alpha)$. 
			If $\alpha\in\Lambda$ and $|\alpha|>\max\{|z_1-z_2|\mid z_1,z_2\in \partial T,\ T\in\tiling\}$ (where $\partial T$ denotes the boundary of tile $T$), then 
			$\stiling=\{\cmap(T)\mid T\in\tiling\}$ is a tiling of $\C^{\ast}$.
		\end{thm}
		
		\begin{proof}
			Let $T\in\tiling$. 
			Clearly, $\cmap(\partial T)$ is a closed curve.
			Suppose $z_1,z_2\in\partial T$ such that $\cmap(z_1)=\cmap(z_2)$. Then $z_1=z_2+k\alpha$ for some $k\in\Z$ by Lemma \ref{lem:SameImage}. 
			We have $|\alpha|>|z_1-z_2|=|k||\alpha|$ which implies that $k=0$ and $z_1=z_2$.
			Thus, $\cmap$ maps $\partial T$ one-to-one and onto $\cmap(\partial T)$.
			This means that $\cmap(\partial T)$ is simple, and it follows from~\cite[Chapter 3]{Bieb} that $\interior(T)$ is also mapped one-to-one 
			and onto $\interior(\cmap(\partial T))$.  This implies that $\cmap(T)$ is a tile.
			
			Since $\tiling$ is a covering of $\C$ and $\cmap$ is onto, $\mathscr{S}_{\alpha}$ is a covering of $\C^{\ast}$.
			All that remains is to show that $\stiling$ is a packing of~$\C^{\ast}$.
			Indeed, let $T_1,T_2\in\tiling$ such that $\interior(\cmap(T_1))\cap\interior(\cmap(T_2))\neq\varnothing$. 
			Then there exist $z_1\in\interior(T_1)$ and $z_2\in\interior(T_2)$ such that  
			$z_2=z_1+k\alpha$ for some $k\in\Z$ by Lemma \ref{lem:SameImage}. 
			Since $\alpha\in\Lambda$, we have $z_2\in\interior(T_1+k\alpha)$. 
			Hence, $T_2=T_1+k\alpha$ because $\tiling$ is a packing of $\C$, and
			thus, $\cmap(T_1)=\cmap(T_2)$.
		\end{proof}
		
		Henceforth, it shall be assumed that $\alpha$ satisfies the required conditions stated in Theorem \ref{alphaconditions} so that $\stiling$ is a tiling of $\C^{\ast}$. 
		We say that such an $\alpha$ is \emph{admissible}.
		We make the following observations about admissible values of $\alpha$.
		\begin{enumerate}
			\item Since $\alpha\in\Lambda$, there exist unique integers $L$ and $R$ such that $\alpha =L+Ri$, $\alpha=L+R\omega$ and $\alpha=(L+R\omega)(2+\omega)=(2L-R)+(L+R)\omega$ 
			for the $(4^4)$, $(6^3)$ and $(3^6)$ tilings, respectively. 
			We now say that $L$ and $R$ are the \emph{integers that determine} $\alpha$.
			
			\item Since $|\alpha|>\max\{|z_1-z_2|\mid z_1,z_2\in \partial T,\ T\in\tiling\}$, $|\alpha|$ is greater than $\sqrt{2}$, $\frac{2\sqrt{3}}{3}$ and $\sqrt{3}$ for the 
			$(4^4)$, $(6^3)$ and $(3^6)$ tilings, respectively. 
		\end{enumerate}
	
		The following proposition states that if $\alpha$ is an admissible value, then replacing $\alpha$ by any of its associates in \eqref{eq:ConformalMap} yields the same tiling 
		$\stiling$.  Thus, to exhaust all possible $\stiling$, it is enough to consider one representative from each associate class of $\Z[\xi]$. 

		\begin{prop}\label{prop:SameImage}
			Let $\varphi_{\alpha_1}(z)=\exp(2\pi i z/\alpha_1)$,
			$\varphi_{\alpha_2}(z)=\exp(2\pi i z/\alpha_2)$, and $\xi\in\{i,\omega\}$.
			If $\alpha_1$ and $\alpha_2$ are associates in $\Z[\xi]$, then $\mathscr{S}_{\alpha_1}=\mathscr{S}_{\alpha_2}$.
		\end{prop}		
		\begin{proof}
			Let $\xi =i$ if $\tiling$ is the $(4^4)$ tiling, and $\xi=\omega$ if $\tiling$ is the $(6^3)$ or $(3^6)$ tiling.
			If $\alpha_1$ and $\alpha_2$ are associates, then $\alpha_2=\epsilon\alpha_1$ for some unit $\epsilon\in\Z[\xi]$. Now $f(z)=\epsilon z$ is a symmetry of $\tiling$, so  
			$f(T)=\epsilon T\in\tiling$ whenever $T\in\tiling$. The conclusion then follows since $\varphi_{\alpha_1}(T)=\varphi_{\alpha_2}(\epsilon T)$. 
		\end{proof}
		
		We now describe the tilings $\stiling$.  
		Note that the map $\cmap$ sends a line in $\C$ to one of the following: a logarithmic spiral with the asymptotic point at the origin, a circle centered at the origin, and an 
		open-ended ray emanating from the origin.  This is the reason why $\stiling$ ends up having a singular point at the origin.
		
		The lines (or edges, in the case of the $(3^6)$ tiling) of $\tiling$ can be partitioned into equivalence classes, where two lines belong to the same equivalence class if
		they are parallel. 
		There are two equivalence classes of parallel lines for the $(4^4)$ tiling, and three for the $(6^3)$ and $(3^6)$ tilings.
		It is easy to verify that the images of lines that belong to the same equivalence class are of the same type.
		
		The image under $\cmap$ of an equivalence class of parallel lines in $\tiling$ depends on the prime decomposition of $\alpha$ in $\Z[\xi]$, where $\xi=i$ if $\tiling$ is
		the $(4^4)$ tiling and $\xi=\omega$ if $\tiling$ is the $(6^3)$ or $(3^6)$ tiling.  Hence,
		we categorize all possible $\stiling$ into three classes.  Class 1 tilings are those $\stiling$ for which $\alpha$ is not balanced.  The tiling $\stiling$ falls under Class 2 
		whenever $\alpha$ is balanced and the factorization of $\alpha$ contains an odd power of a factor of the ramified prime in $\Z[\xi]$.
		In this case, we write $\alpha=\epsilon(L-L\xi)$ for some unit $\epsilon$ in $\Z[\xi]$ and $L\in\Z$.  Finally, if $\alpha$ is an integer multiple of a unit in $\Z[\xi]$,
		that is, $\alpha=\epsilon L$, then we say that $\stiling$ belongs to Class 3.  We summarize the curves that bound the tiles of $\stiling$ for each class in Table \ref{stiling}.		
		
		\begin{table}[ht]
			\begin{tabular}{|c||c|c|}\hline
				Class & $\tiling$ & Images of lines (or edges) of $\tiling$ under $\cmap$\\\hline 
				\multirow{3}{*}{1} & $(4^{4})$ & 
				\begin{minipage}[c]{0.82\linewidth}
					two sets of spirals that contain $|L|$ and $|R|$ spirals, spirals in one set are positively-oriented while spirals in the other set are negatively-oriented
				\end{minipage}\\\cline{2-3}
				& $(6^{3})$ & 
				\begin{minipage}[c]{0.82\linewidth}
					three sets of spirals that contain $|2L-R|$, $|2R-L|$ and $|L+R|$ spirals, spirals in two sets go in the same direction while spirals in the third set go in the
					opposite direction
				\end{minipage}\\\cline{2-3} 
				& $(3^{6})$ & 
				\begin{minipage}[c]{0.82\linewidth}
					three sets of spirals that contain $|L|$, $|R|$ and $|L-R|$ spirals, spirals in two sets go in the same direction while spirals in the third set go in the
					opposite direction
				\end{minipage}\\\hline\hline
				\multirow{3}{*}{2} & $(4^{4})$ & 
				\begin{minipage}[c]{0.82\linewidth}
					$|L|$ positively- and $|L|$ negatively-oriented spirals
				\end{minipage}\\\cline{2-3}
				& $(6^{3})$ & 
				\begin{minipage}[c]{0.82\linewidth}
					$|3L|$ positively- and $|3L|$ negatively-oriented spirals, circles centered at the origin 
				\end{minipage}\\\cline{2-3}
				& $(3^{6})$ & 
				\begin{minipage}[c]{0.82\linewidth}
					$|L|$ positively- and $|L|$ negatively-oriented spirals, circles centered at the origin 
				\end{minipage}\\\hline\hline
				\multirow{3}{*}{3} & $(4^{4})$ & 
				\begin{minipage}[c]{0.82\linewidth}
					$|L|$ rays emanating from the origin, circles centered at the origin
				\end{minipage}\\\cline{2-3}						
				& $(6^{3})$ & 
				\begin{minipage}[c]{0.82\linewidth}
					$|L|$ positively- and $|L|$ negatively-oriented spirals, $|2L|$ rays emanating from the origin 
				\end{minipage}\\\cline{2-3}						
				& $(3^{6})$ & 
				\begin{minipage}[c]{0.82\linewidth}
					$|L/3|$ positively- and $|L/3|$ negatively-oriented spirals, $|2L/3|$ rays emanating from the origin 
				\end{minipage}\\\hline		
			\end{tabular}
			\caption{Three Classes of Tilings with a Singularity at the Origin and the Curves that Bound Their Tiles}
			\label{stiling}
		\end{table}
		
		Let $[\ell]$ be an equivalence class of lines (or edges) of $\tiling$.  Observe from Table \ref{stiling} that if $\cmap([\ell])$ is a set of logarithmic spirals or a set of rays, 
		then it contains only a finite number of elements.  In addition, as we will see in the next section, the spirals and rays are spread out symmetrically about the origin.  
		On the other other hand, if $\cmap([\ell])$ is a set of circles, then it is countably infinite.  Here, the radii of the circles increase geometrically.  
		Figure \ref{fig:(4^4)Classes} shows an example from each class for the $(4^4)$ tiling.
				
		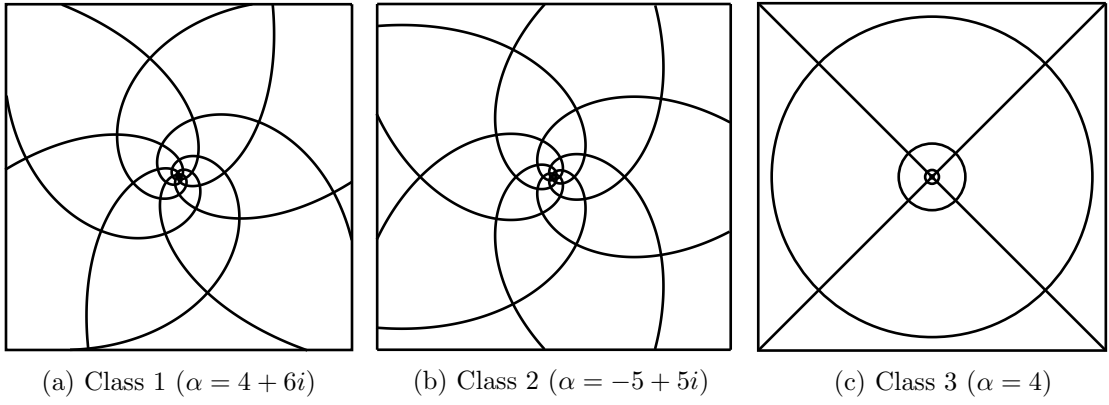
\begin{figure}[ht]
			\begin{subfigure}[c]{0.3\linewidth}
				\centering
				\begin{tikzpicture}[scale=1.15]
					\draw[line width=1pt, color=black, domain=-2.5:1.08,samples=200,smooth] plot (xy polar cs:angle=\x r,radius={exp(1.5*(\x)-0.79)});
					\draw[line width=1pt, color=black, domain=-2.5:2.17,samples=200,smooth] plot (xy polar cs:angle=\x r,radius={exp(1.5*(\x)-2.37)});
					\draw[line width=1pt, color=black, domain=-2.5:3.10,samples=200,smooth] plot (xy polar cs:angle=\x r,radius={exp(1.5*(\x)-3.95)});
					\draw[line width=1pt, color=black, domain=-2.5:4.23,samples=200,smooth] plot (xy polar cs:angle=\x r,radius={exp(1.5*(\x)-5.53)});
					\draw[line width=1pt, color=black, domain=-2.5:5.35,samples=200,smooth] plot (xy polar cs:angle=\x r,radius={exp(1.5*(\x)-7.11)});
					\draw[line width=1pt, color=black, domain=-1:6.26,samples=200,smooth] plot (xy polar cs:angle=\x r,radius={exp(1.5*(\x)-8.69)});
					\draw[line width=1pt, color=black, domain=-0.35:10,samples=200,smooth] plot (xy polar cs:angle=\x r,radius={exp(-0.67*(\x)+0.52)});
					\draw[line width=1pt, color=black, domain=1.17:10,samples=200,smooth] plot (xy polar cs:angle=\x r,radius={exp(-0.67*(\x)+1.56)});
					\draw[line width=1pt, color=black, domain=2.7:10,samples=200,smooth] plot (xy polar cs:angle=\x r,radius={exp(-0.67*(\x)+2.6)});
					\draw[line width=1pt, color=black, domain=4.15:10,samples=200,smooth] plot (xy polar cs:angle=\x r,radius={exp(-0.67*(\x)+3.64)});
					\draw[line width=1] (2,2)--(-2,2)--(-2,-2)--(2,-2)--(2,2);
				\end{tikzpicture}
				\caption{Class 1 ($\alpha=4+6i$)}
			\end{subfigure}
			\begin{subfigure}[c]{0.3\linewidth}
				\begin{tikzpicture}[scale=.5]
					\draw[line width=1pt, color=black, domain=-4:1.04,samples=200,smooth] plot (xy polar cs:angle=\x r,radius={exp((\x)+0.63)});
					\draw[line width=1pt, color=black, domain=-5:-0.3,samples=200,smooth] plot (xy polar cs:angle=\x r,radius={exp((\x)+3*0.63)});
					\draw[line width=1pt, color=black, domain=-7.14:-1.625,samples=200,smooth] plot (xy polar cs:angle=\x r,radius={exp((\x)+5*0.63)});
					\draw[line width=1pt, color=black, domain=-8.14:-2.8,samples=200,smooth] plot (xy polar cs:angle=\x r,radius={exp((\x)+7*0.63)});
					\draw[line width=1pt, color=black, domain=-10.28:-3.845,samples=200,smooth] plot (xy polar cs:angle=\x r,radius={exp((\x)+9*0.63)});
					\draw[line width=1pt, color=black, domain=-1.04:4,samples=200,smooth] plot (xy polar cs:angle=\x r,radius={exp(0.63-(\x))});
					\draw[line width=1pt, color=black, domain=0.3:5,samples=200,smooth] plot (xy polar cs:angle=\x r,radius={exp((3*0.63)-(\x))});
					\draw[line width=1pt, color=black, domain=1.625:7.14,samples=200,smooth] plot (xy polar cs:angle=\x r,radius={exp((5*0.63)-\x)});
					\draw[line width=1pt, color=black, domain=2.8:8.14,samples=200,smooth] plot (xy polar cs:angle=\x r,radius={exp((7*0.63)-(\x))});
					\draw[line width=1pt, color=black, domain=3.845:10.28,samples=200,smooth] plot (xy polar cs:angle=\x r,radius={exp((9*0.63)-(\x))});
					\draw[line width=1pt, color=black] (-4.7,4.6)--(4.7,4.6);
					\draw[line width=1pt, color=black] (-4.7,-4.6)--(-4.7,4.6);
					\draw[line width=1pt, color=black] (-4.7,-4.6)--(4.7,-4.6);
					\draw[line width=1pt, color=black] (4.7,-4.6)--(4.7,4.6);
				\end{tikzpicture}
				\caption{Class 2 ($\alpha=-5+5i$)}
			\end{subfigure}
			\begin{subfigure}[c]{0.3\linewidth}
				\begin{tikzpicture}[scale=.042]
					\draw[line width=1pt] (0,0) circle (0.46cm);
					\draw[line width=1pt] (0,0) circle (2.19cm);
					\draw[line width=1pt] (0,0) circle (10.55cm);
					\draw[line width=1pt] (0,0) circle (50.75cm);
					\draw[line width=1pt] (55,55)--(-55,55)--(-55,-55)--(55,-55)--(55,55);
					\draw[line width=1pt] (55,55)--(-55,-55);
					\draw[line width=1pt] (-55,55)--(55,-55);
				\end{tikzpicture}
				\caption{Class 3 ($\alpha=4$)}
			\end{subfigure}
			\caption{Tilings with singularity obtained by applying the conformal map $\cmap$ on the $(4^4)$ tiling}
			\label{fig:(4^4)Classes}
		\end{figure}			

	\section{Symmetries of $\stiling$}\label{sec:SymmetryGroup}
		Since $\stiling$ has only one singular point, a symmetry of $\stiling$ must fix that point.
		Thus, $\sg(\stiling)$ is a group of isometries that fix the origin, and consequently, it must be isomorphic to some finite cyclic group $C_n$ or finite dihedral group
		$D_n$.  This means that to identify $\sg(\stiling)$, we only need to find a generator of $\rg(\stiling)$ and to determine whether any reflection symmetry leaves $\stiling$ 
		invariant.  For this, we turn to the symmetries of $\tiling$.
		
		\begin{lem}\label{lem:CorrespondingSymm}
			Let $\tiling$ be a regular tiling, $\cmap(z)=\exp(2\pi i z/\alpha)$ with $\alpha$ admissible, and $\stiling=\{\cmap(T)\mid T\in\tiling\}$.
			If $g$ is an isometry of $\C^{\ast}$ such that there exists a symmetry $f\in\sg(\tiling)$ for which $g\cmap=\cmap f$, then $g\in\sg(\stiling)$.
		\end{lem}
		\begin{proof}
			We have $g(\stiling)=g(\cmap(\tiling))=\cmap(f(\tiling))=\cmap(\tiling)=\stiling$.
		\end{proof}
		
		If $g\in\sg(\stiling)$ and $f\in\sg(\tiling)$ for which the following diagram commutes, 
		\begin{center}
			\begin{tikzpicture}
				\matrix (m) [matrix of math nodes, row sep=3.5em, column sep=5em, text height=2ex, text depth=0.25ex]
				{\C & \C^{\ast}  \\
					\C & \C^{\ast} \\};
				\path[-stealth]
					(m-1-1) edge node [left] {$f$} (m-2-1)
					(m-2-1) edge node [above] {$\cmap$} (m-2-2) 
					(m-1-1) edge node [above] {$\cmap$} (m-1-2)
				(m-1-2) edge node [right] {$g$} (m-2-2);
			\end{tikzpicture}
		\end{center}		
		then we say that $f$ \emph{corresponds} to $g$.  Note that  an $f\in\sg(\tiling)$ that corresponds to $g\in\sg(\stiling)$ is not unique. Indeed, if 
		$h=f+k\alpha\in\sg(\tiling)$ for some $k\in\Z$, then $h$ also corresponds to $g$.  We now use Lemma \ref{lem:CorrespondingSymm} to find
		the symmetry group of  $\stiling$.
		
		\begin{thm}\label{thm:SymmGroupofS}
				Let $\tiling$ be a regular tiling, $\cmap(z)=\exp(2\pi i z/\alpha)$ with $\alpha$ admissible, and $\stiling=\{\cmap(T)\mid T\in\tiling\}$.
				Suppose $n=\gcd(L,R)$, where $L$ and $R$ are the integers that determine $\alpha$.
				Define $g_1(z)=\exp(2\pi i/n)z$ and $g_2(z)=\bar{z}$.
				\begin{enumerate}[\emph{(}a\emph{)}]
					\item If $\alpha$ is not balanced, then $\sg(\stiling)=\langle g_1\rangle\cong C_n$.
					
					\item If $\alpha$ is balanced, then $\sg(\stiling)=\langle g_1,g_2\rangle\cong D_n$.
				\end{enumerate}
		\end{thm}
		\begin{proof}
			Let $m\in\N$ such that $m\mid\alpha$.  Then $g(z)=\exp(2\pi i/m)z\in\sg(\stiling)$, and a symmetry of $\tiling$ that corresponds to $g$ is the translation given by 
			$f(z)=z+\frac{\alpha}{m}$.  Hence, the rotation group $\rg(\stiling)$ is generated by $g_1$.		
			
			If  $\alpha$ is not balanced, then $\stiling$ falls under Class 1.  Since the number of positively-oriented spirals and the number of negatively-oriented spirals are not the
			same  (see Table~\ref{stiling}), $\stiling$ cannot have a reflection symmetry.  
			
			If $\alpha$ is balanced, then $g_2$ leaves $\stiling$ invariant.  Indeed, a symmetry of $\tiling$ that corresponds to $g_2$ is the reflection $f_2$ given by 
			$f_2(z)=\epsilon\bar{z}$ for some unit  $\epsilon$ of $\Z[\xi]$ (where $\xi=i$ if $\tiling$ is the $(4^4)$ tiling and $\xi=\omega$ if $\tiling$ is the $(6^3)$ or $(3^6)$ tiling).
		\end{proof}
		
		Not only does Theorem~\ref{thm:SymmGroupofS} give us the symmetry group of $\stiling$, but it also tells us that for any positive integer $n$, it is possible to construct a 
		tiling $\stiling$ that has rotation symmetry of order $n$ about the origin by choosing a suitable $\alpha$.   Moreover,  we obtain that for every $g\in\sg(\stiling)$, there exists 
		$f\in\sg(\tiling)$ such that $f$ corresponds to $g$.
		
		In fact, the tilings $\stiling$ are not only symmetrical, but also self-similar whose scaling factor is some power of $e$.  This has also been observed in \cite{Luck}, where 
		some tilings with a singularity even have scaling factors related to quasiperiodic patterns such as the golden ration $\tau$.  However, this aspect of $\stiling$ will not be 
		discussed here but will be published elsewhere.
		
		There is a one-to-one correspondence between a regular tiling $\tiling$ and the set $P$ of centers of tiles of $\tiling$. 
		Furthermore, $\tiling$ and $P$ have the same symmetry group, whose action on $\tiling$ is equivalent to its action on $P$.
		This means that every coloring of $\tiling$ may be viewed as a coloring of $P$, and both colorings have the same color symmetry group.
		We now proceed to show a similar relationship between $\stiling$ and $\qalpha:=\{\cmap(t)\mid t\in P\}$ .  
		
		\begin{lem}\label{thm:1-1CorrS&Q}
			Let $\tiling$ be a regular tiling, $P$ be the set of centers of $\tiling$, and $\cmap(z)=\exp(2\pi i z/\alpha)$ with $\alpha$ admissible.
			Suppose $\stiling=\{\cmap(T)\mid T\in\tiling\}$, $\qalpha=\{\cmap(t)\mid t\in P\}$, and
			$\lambda:\tiling\rightarrow P$ is the map that sends a tile of $\tiling$ to its center. 
			Then the map $\sigma:\stiling\rightarrow\qalpha$ that sends the image of a tile under $\cmap$ to the image of its center under $\cmap$, that is, 
			$(\sigma\cmap)(T)=(\cmap\lambda)(T)$ for $T\in\tiling$, defines a one-to-one correspondence between $\stiling$ and $Q_{\alpha}$.
		\end{lem}		
		\begin{proof}
			Suppose that $T_1,T_2\in\tiling$ such that $\cmap(T_1)=\cmap(T_2)$.
			Then there exists $t'\in T_1$ such that $\cmap(t')=\cmap(\lambda(T_2))$. 
			By Lemma \ref{lem:SameImage}, $t'=\lambda(T_2)+k\alpha$ for some $k\in\Z$.
			Since $\lambda(T_2)\in P$ and $\alpha\in\Lambda$, it follows that $t'=\lambda(T_1)$. 
			Consequently, $\sigma$ is well-defined. 
			
			Now assume that $\sigma(\cmap(T_1))=\sigma(\cmap(T_2))$.
			Then $\cmap(\lambda(T_1))=\cmap(\lambda(T_2))$.
			This implies that $\lambda(T_1)=\lambda(T_2)+k\alpha$ for some $k\in\Z$. 
			Since $\alpha\in\Lambda$, $T_1=T_2+k\alpha$ and we obtain $\cmap(T_1)=\cmap(T_2)$. 
			Thus, $\sigma$ is injective.  Clearly, $\sigma$ is surjective since $\lambda$ is surjective.
		\end{proof}
		
		Lemma \ref{thm:1-1CorrS&Q} gives us the following commutative diagram.
		\begin{center}
			\begin{tikzpicture}
				\matrix (m) [matrix of math nodes, row sep=3.5em, column sep=5em, text height=2ex, text depth=0.25ex]
				{\tiling & \stiling  \\
					P & \qalpha \\};
				\path[-stealth]
					(m-1-1) edge node [left] {$\lambda$} (m-2-1)
					(m-2-1) edge node [above] {$\cmap$} (m-2-2) 
					(m-1-1) edge node [above] {$\cmap$} (m-1-2)
				(m-1-2) edge node [right] {$\sigma$} (m-2-2);
			\end{tikzpicture}
		\end{center}
		
		\begin{thm}\label{thm:SameSymmetryGroup}
			Let $P$ be the set of centers of a regular tiling $\tiling$.  Suppose $\cmap(z)=\exp(2\pi i z/\alpha)$ with $\alpha$ admissible,
			$\stiling=\{\cmap(T)\mid T\in\tiling\}$, and $\qalpha=\{\cmap(t)\mid t\in P\}$.
			Then $\sg(\stiling)\subseteq\sg(\qalpha)$ and the action of $\sg(\stiling)$ on $\stiling$ and $\qalpha$ are equivalent.
		\end{thm}		
		\begin{proof}
			Let $g\in\sg(\mathscr{S}_{\alpha})$ and $\varphi_{\alpha}(T)\in\mathscr{S}_{\alpha}$. 
			Suppose $f\in\sg(\mathscr{T})$ corresponds to $g$. 
			Then $g(\qalpha)=\cmap(f(P))=\cmap(P)=\qalpha$, and so $\sg(\stiling)\subseteq\sg(\qalpha)$.  
			Let $\sigma$ be as defined in Lemma \ref{thm:1-1CorrS&Q}.
			Since the action of $f$ on $\mathscr{T}$ is equivalent to its action on $P$, we have 
			\[(g\sigma)(\cmap(T))=g(\cmap(\lambda(T)))=\varphi_{\alpha}(f(\lambda(T)))=
				\varphi_{\alpha}(\lambda(f(T))=\sigma(\varphi_{\alpha}(f(T)))=(\sigma g)(\varphi_{\alpha}(T)).\]
			This shows that the action of $\sg(\stiling)$ on $\stiling$ is equivalent to its action on $\qalpha$.
		\end{proof}
		
		Hence, there is a one-to-one correspondence between colorings of $\stiling$ and colorings of $\qalpha$ via the map $\sigma$.  Furthermore, elements of $\sg(\stiling)$ that 
		are color symmetries of a coloring of $\qalpha$ are also color symmetries of the associated coloring of $\stiling$. 
		
	\section{Color Symmetries of $\stiling$}\label{sec:ColorSymmetry}
		Our aim in this section is to come up with chirally and fully perfect colorings of the tiling $\stiling=\cmap(\tiling)$ that has a singularity at the origin.
		In Lemma \ref{thm:1-1CorrS&Q}, we defined $\lambda$ to be the bijection that sends a tile $T\in\tiling$ to its center $t\in P$, 
		and $\sigma$ to be the bijection that sends $\cmap(T)\in \stiling$ to $\cmap(\lambda(T))\in\qalpha$.
		If $\coloring_P=\{X_i\}_{i=1}^m$ is a coloring of $P$, then $\coloring_{\tiling}=\{\lambda^{-1}(X_i)\}_{i=1}^{m}$ gives a coloring of $\tiling$.
		Similarly, if $\coloring_{\qalpha}=\{Y_i\}_{i=1}^k$ is a coloring of $\qalpha$, then $\coloring_{\stiling}=\{\sigma^{-1}(Y_i)\}_{i=1}^k$ yields a coloring of $\stiling$. 
		In addition, Theorem \ref{thm:SameSymmetryGroup} assures us that if the elements of $\rg(\stiling)$ (or $\sg(\stiling)$) permute the colors of $\coloring_{\qalpha}$,
		then $\coloring_{\stiling}$ is a chirally (or fully) perfect coloring.  All that remains is to obtain colorings of $\qalpha$ from colorings of $P$.
		
		\subsection{Ideal Colorings of $P$}		
			We start by identifying the chirally and fully perfect colorings of the set $P$ of centers of the regular tiling $\tiling$.			
			If $\tiling$ is the $(4^4)$ or the $(6^3)$ tiling, then $P=\Z[i]$ or $P=\Z[\omega]$.  Hence, Theorem \ref{thm:PerfectT} applies and
			we already have the chirally and fully perfect colorings of $P$.
			
			In the case where $\tiling$ is the $(3^6)$ tiling, $P$ is not a lattice. 
			Nonetheless, a coloring of $P$ may still be obtained from colorings of the lattice $\Z[\omega]$.
			Indeed, let $\coloring_{\Z[\omega]}=\{ X_{i}\} _{i=1}^{m}$ be the coloring of $\Z[\omega]$ induced by a sublattice $\Gamma$ of $\Z[\omega]$.
			Since $P\subset\Z[\omega]$, we obtain a coloring of $P$ from $\coloring_{\Z[\omega]}$ given by
			\[\coloring_P=\{X_i\cap P \mid X_i\cap P\neq\varnothing,\:1\leq i\leq m\}.\]
			When needed, we reindex so that $\coloring_P=\left\{ X_{i}\cap P\right\} _{i=1}^{m'}$ for some $m'\leq m$. 
			In fact, $m'$ is either $m$ or $\frac{2}{3}m$, as shown in the next lemma.
			
			\begin{lem}\label{thm:No.ofColors(3^6)}
				Let $\Gamma$ be a sublattice of index $m$ in $\Z[\omega]$, and 
				$\coloring_{\Z[\omega]}=\{X_i\}_{i=1}^m$ be the coloring of $\Z[\omega]$ induced by $\Gamma$.
				\begin{enumerate}[\emph{(}a\emph{)}]
					\item If $\Gamma$ is a sublattice of $P_0=(2+\omega)$, then the coloring of $P=\Z[\omega]\setminus P_0$ obtained from $\coloring_{\Z[\omega]}$ is 
					$\coloring_P=\{X_i\}_{i=1}^{m'}$, where $m'=\frac{2}{3}m$.
					
					\item If $\Gamma$ is not a sublattice of $P_0$, then the coloring of $P$ obtained from $\coloring_{\Z[\omega]}$ is $\coloring_P=\{X_i\cap P\}_{i=1}^{m}$. 
				\end{enumerate} 
			\end{lem}			
			\begin{proof}
				Suppose $\Gamma$ is a sublattice of $P_0$, and $X_i\in\coloring_{\Z[\omega]}$.
				If $X_i$ contains an element of $P_0$, then $X_i\subseteq P_0$ and consequently, 
				$X_i\cap P=\varnothing$.  Otherwise, $X_i\cap P=X_i$.  Since $m=[\Z[\omega]:P_0][P_0:\Gamma]=3[P_0:\Gamma]$, 
				the number of cosets of $\Gamma$ that lie entirely in $P$ is $m-[P_0:\Gamma]=\frac{2}{3}m$.  Therefore, $\coloring_P=\{X_i\}_{i=1}^{m'}$, where $m'=\frac{2}{3}m$.
				
				If $\Gamma$ is not a sublattice of $P_0$, then every $X_i\in\coloring_{\Z[\omega]}$ must contain an element which is not in $P_0$.  
				Hence, $X_i\cap P\neq \varnothing$ for all $i\in\{1,\ldots,m\}$, which yields the claim.
			\end{proof}
			
			We now establish a relationship between perfect colorings of $\Z[\omega]$ and perfect colorings of $P$.
			Observe that the generators $h_1$, $h_2$, $h_3$, and $h_4$ of $\sg(\mathscr{T})$ (see Table \ref{tab:SymmetryGroupT}) permute $P_0$, $P_1=-1+(2+\omega)$ and $P_2=1+(2+\omega)$.  
			In particular, $h_1(P_0)=P_0$, $h_1(P_1)=P_2$ and $h_1(P_2)=P_1$, while $h_2$, $h_3$, and $h_4$ leave each coset invariant.  We use this fact in the proofs of 
			Lemma \ref{lem:orthogonalpermute} and Theorem \ref{thm:perfectcoloringP}.
			
			\begin{lem}\label{lem:orthogonalpermute}
				Let $\Gamma$ be a sublattice of $\Z[\omega]$ such that $[\Z[\omega]:\Gamma]=m$. 
				Suppose $\coloring_{\Z[\omega]}$ is the coloring of $\Z[\omega]$ induced by $\Gamma$, and $\coloring_P$ is the coloring of $P$ obtained from $\coloring_{\Z[\omega]}$.  
				If $f\in \sg(\mathscr{T})\cap \Orth(2)$ is a color symmetry of $\coloring_P$, then $f$ is a color symmetry of $\coloring_{\Z[\omega]}$.
			\end{lem}			
			\begin{proof}
				Let $\coloring_{\Z[\omega]}=\{X_i\}_{i=1}^m$.  It follows from \cite[Theorem 2]{Pena} that $f$ is a color symmetry of $\coloring_{\Z[\omega]}$ whenever $f$ 
				leaves $\Gamma$ invariant.
			
				First, we consider the case when $\Gamma$ is a sublattice of $P_0$. Then $\coloring_P=\{X_i\}_{i=1}^{m'}$.  By assumption, given $X_i\in\coloring_P$,
				$f(X_i)=X_j$ for some $j\in\{1,\ldots,m'\}$.  If we write $X_i=t_i+\Gamma$ and $X_j=t_j+\Gamma$, then for all $\ell\in\Gamma$,
				$f(t_i+\ell)\in t_j+\Gamma$.  This implies that $f(\ell)\in (t_j-f(t_i))+\Gamma=\Gamma$.  Hence, $\Gamma$ contains $f(\Gamma)$ and so $f(\Gamma)=\Gamma$ because $f$ is 
				an isometry.
				
				Now suppose $\Gamma$ is not a sublattice of $P_0$.  We have $\coloring_P=\{X_i\cap P\}_{i=1}^{m}$. 
				Let $t_0\in\Gamma\cap P_0$.  Then $f(t_0)\in P_0$.  Given $t\in\Gamma\cap P$, we must have  $t_0+t\in\Gamma\cap P$.
				Since $f$ permutes the colors of $\coloring_P$, $f(t)$ and $f(t_0+t)$ must belong to the same coset of $\Gamma$.
				We then have $f(t_0)=f(t_0+t)-f(t)\in\Gamma$.  Thus, $f(\Gamma\cap P_0)\subseteq\Gamma\cap P_0$, and it follows that 
				$f(\Gamma\cap P_0)=\Gamma\cap P_0$. 				
				We obtain $f(\Gamma)=(\Gamma\cap P_0)\cup(X_j\cap P)$, for some $j\in\{1,\ldots,m\}$. 
				However, $f(\Gamma)$ must be a sublattice of $\Z[\omega]$, so $X_j=\Gamma$ and $f(\Gamma)=\Gamma$.
			\end{proof}
			
			\begin{thm}\label{thm:perfectcoloringP}
				Let $\Gamma$ be a sublattice of index $m$ in $\Z[\omega]$.
				Then the coloring $\coloring_{\Z[\omega]}$ of $\Z[\omega]$ induced by $\Gamma$ is chirally (respectively, fully) perfect if and only if the coloring $\coloring_P$ 
				of $P$ obtained from $\coloring_{\Z[\omega]}$ is chirally (respectively, fully) perfect. 
			\end{thm}			
			\begin{proof}
				If $h_1$ permutes (and $h_2$ permutes) the colors of $\coloring_P$, then Lemma \ref{lem:orthogonalpermute} guarantees that $h_1$ permutes (and $h_2$ permutes) the colors 
				of $\coloring_{\Z[\omega]}$.
				Since $h_3$ and $h_4$ are translation symmetries of $\Z[\omega]$, they must be color symmetries of $\coloring_{\Z[\omega]}$. 
				This proves the backward direction.
				
				Suppose $\coloring_{\Z[\omega]}=\{X_i\}_{i=1}^m$ is chirally (fully) perfect.  Then given $X_i\in \coloring_{\Z[\omega]}$ for which $X_i\cap P\neq\varnothing$,
				$h_1(X_i)=X_j$ for some $j\in\{1,\ldots,m\}$.  
				We then have $h_1(X_i\cap P)=X_j\cap P\neq\varnothing$.
				Thus, $h_1$ permutes the colors of $\coloring_P$. 
				In a similar manner, we can show that $h_3$, $h_4$ (and $h_2$) are also color symmetries of $\coloring_P$. 
			\end{proof}
			
			Hence, even if $\tiling$ is the $(3^6)$ tiling, ideal colorings of $\Z[\omega]$ give rise to chirally or fully perfect colorings of $P$.
			Conversely, if a coloring of $P$ obtained from a sublattice coloring of $\Z[\omega]$ is chirally or fully perfect, then the coloring 
			of $\Z[\omega]$ must be an ideal coloring.
		
		\subsection{Compatibility between Sublattice Colorings of $P$ and Conformal Maps $\cmap$}
			Given a coloring $\coloring_P=\{X_i\} _{i=1}^{m}$ of $P$, the set $\{\cmap(X_i)\}_{i=1}^{m}$ does not necessarily yield a coloring of $\qalpha$.
			The problem is since $\cmap$ is not one-to-one, it is possible that $\cmap(X_i)\cap\cmap(X_{j})\neq\varnothing$ for some $i,j\in\{1,\ldots,m\}$ and $i\neq j$.
			In that case, two points of $P$ that are assigned different colors in $\coloring_P$ have the same image under $\cmap$. 
			If it happens that $\{\cmap(X_i)\} _{i=1}^{m}$ is a coloring of $\qalpha$, then we say that $\coloring_P$ is \emph{compatible with} $\cmap$. 
			That is, each point of $\qalpha$ is assigned exactly one color when $\cmap$ is applied on $\coloring_P$.
			We denote the coloring of $\qalpha$ obtained in this manner by $\cmap(\coloring_P)$.
			
			We now characterize the sublattice colorings of $P$ that are compatible with $\cmap$. 
			
			\begin{thm}\label{thm:CompatibleSublattice}
				Let $P$ be the set of centers of a regular tiling, and $\cmap(z)=\exp(2\pi i z/\alpha)$ with $\alpha$ admissible.
				Then the coloring $\coloring_P$ of $P$ induced by a sublattice $\Gamma$ is compatible with $\cmap$ if and only if $\alpha\in\Gamma$.
			\end{thm}			
			\begin{proof}
				Suppose $\coloring_P$ is compatible with $\cmap$.  Let $t\in P$. 
				Since $\alpha\in\Lambda$, $t+\alpha\in P$.
				By Lemma \ref{lem:SameImage}, $\cmap(t)=\cmap(t+\alpha)$ and so $t$ and $t+\alpha$ must belong to the same coset of $\Gamma$.
				Hence, $\alpha=(t+\alpha)-t\in\Gamma$.
				
				Conversely, let $t,t'\in P$ such that $\cmap(t)=\cmap(t')$.  This means that for some $k\in\Z$, $t'-t=k\alpha\in\Gamma$ since $\alpha\in\Gamma$. 
				Thus, $t'$ and $t$ are in the same coset of $\Gamma$.
			\end{proof}
			
			The following is immediate from Theorem \ref{thm:CompatibleSublattice}.			
			\begin{cor}\label{cor:CompatibleIdeal}
				Let $P$ be the set of centers of a regular tiling, $\coloring_P$ be the coloring of $P$ induced by an ideal $(\beta)$, and $\cmap(z)=\exp(2\pi i z/\alpha)$ with $\alpha$
				admissible.  Then $\coloring_P$ is compatible with $\cmap$ if and only if $(\alpha)\subseteq(\beta)$.
			\end{cor}
			
			Hence, given $\cmap$, we can identify all ideal colorings of $P$ that are compatible with $\cmap$ by looking at the prime decomposition of $\alpha$ in $\Z[\xi]$, 
			where $\xi=i$ if $\tiling$ is the $(4^4)$ tiling and $\xi=\omega$ if $\tiling$ is the $(6^3)$ or $(3^6)$ tiling.
		
		\subsection{Perfect Colorings of $\stiling$}		
			Let $\coloring_P$ be a coloring of $P$ that is compatible with $\cmap$. 
			We now give a necessary and sufficient condition for $g\in\sg(\stiling)$ to be a color symmetry of the coloring $\cmap(\coloring_P)$ of $\qalpha$.
			
			\begin{thm}\label{lem:CorrespondPermutes}
				Let $P$ be the set of centers of a regular tiling, and $\cmap(z)=\exp(2\pi i z/\alpha)$ with $\alpha$ admissible. 
				Suppose a coloring $\coloring_P$ of $P$ is compatible with $\cmap$.
				If $g\in\sg(\stiling)$ and $f\in\sg(\tiling)$ corresponds to $g$,
				then $g$ is a color symmetry of the coloring $\cmap(\coloring_P)$ of $\qalpha$ if and only if $f$ is a color symmetry of $\coloring_P$.
			\end{thm}			
			\begin{proof}
				Suppose $g$ is a color symmetry of $\cmap(\coloring_P)$. 
				If $X_i\in\coloring_P$, then $\cmap(f(X_i))=g(\cmap(X_i))=\cmap(X_j)$ for some $X_j\in\coloring_P$.
				It follows from Lemma \ref{lem:SameImage} that $f(X_i)=X_j+k\alpha$ for some $k\in\Z$. 
				Since $\cmap$ is compatible with $\coloring_P$ and $\cmap(X_j)=\cmap(X_j+k\alpha)$, we have $f(X_i)=X_j$.  
				Thus, $f$ permutes the colors of $\coloring_P$.
				
				In the other direction, let $f$ be a color symmetry of $\coloring_P$.  If $X_i\in\coloring_P$ then $g(\cmap(X_i))=\cmap(f(X_i))=\cmap(X_{j})$ for some $
				X_j\in\coloring_P$. Thus, $g$ permutes the colors of $\cmap(\coloring_P)$.				
			\end{proof}
			
			Theorem \ref{lem:CorrespondPermutes} provides a link between color symmetries of colorings of $P$ and symmetries of $\stiling$ that permute the colors of 
			the corresponding coloring of $\qalpha$.  The next corollary follows from Theorem \ref{thm:PerfectT}, Theorem \ref{thm:SymmGroupofS}, and Theorem \ref{lem:CorrespondPermutes}.
			
			\begin{cor}\label{cor:PerfectColoring}
				Let $P$ be the set of centers of a regular tiling, $\cmap(z)=\exp(2\pi i z/\alpha)$ with $\alpha$ admissible, and $\coloring_P$ be the coloring of $P$ induced by an
				ideal $(\beta)$ that is compatible with $\cmap$.  
				\begin{enumerate}[\emph{(}a\emph{)}]
					\item Then the elements of $\rg(\stiling)$ are color symmetries of the coloring $\cmap(\coloring_P)$ of $\qalpha$.
					
					\item If $\alpha$ and $\beta$ are balanced, then the elements of $\sg(\stiling)$ are color symmetries of $\cmap(\coloring_P)$.
				\end{enumerate}
			\end{cor}
			
			We should note here that not all chirally and fully perfect colorings of $\stiling$ arise from ideal colorings of $P$.  In fact, every coloring
			of $P$ induced by a sublattice $\Gamma$ that contains $\alpha$ will yield a chirally perfect coloring of $\stiling$.
			
			The next corollary confirms the observation in \cite{Luck} that there is an upper bound for the number of colors in chirally perfect colorings of $\stiling$ obtained from 
			ideal colorings of $P$.  
			
			\begin{cor}\label{cor:MaxIdealColoring}
				Let $P$ be the set of centers of a regular tiling, and $\cmap(z)=\exp(2\pi i z/\alpha)$ with $\alpha$ admissible.
				Suppose $\coloring_P$ is an ideal coloring of $P$ that is compatible with $\cmap$.
				Then there is an upper bound $M$ for the number of colors in the coloring $\cmap(\coloring_P)$ of $\qalpha$.  In particular, 
				\begin{enumerate}[\emph{(}a\emph{)}]
					\item if $\tiling$ is the $(4^4)$ or $(6^3)$ tiling, then $M=|\alpha|^2$.
					
					\item if $\tiling$ is the $(3^6)$ tiling, then $M=|\alpha|^2$ if $\alpha\notin P_0=(2+\omega)$ and $M=\frac{2}{3}|\alpha|^2$ if $\alpha\in P_0$. 
				\end{enumerate}
				Moreover, there exists a coloring of $\qalpha$ derived from some ideal coloring of $P$ that has exactly $M$ colors which are permuted by elements
				of $\sg(\stiling)$.
			\end{cor}			
			\begin{proof}
				It follows from Corollary \ref{cor:CompatibleIdeal} that the ideal that satisfies the compatibility condition and gives the most number of colors is $(\alpha)$. 
				Then $M=[\Z[\xi]:(\alpha)]=|\alpha|^2$ if $\tiling$ is the $(4^4)$ or $(6^3)$ tiling. 
				Lemma \ref{thm:No.ofColors(3^6)} gives us the value of $M$ if $\tiling$ is the $(3^6)$ tiling. 
				Finally, the elements of $\sg(\stiling)$ permute the colors in $\cmap(\coloring_P)$ obtained from $\coloring_P$ induced by $(\alpha)$ by
				Corollary~\ref{cor:PerfectColoring}.			
			\end{proof}
			
			However, unlike colorings of $\qalpha$ obtained from ideal colorings of $P$, there is no upper bound for the number of colors of colorings of $\qalpha$ obtained from 
			sublattice colorings of $P$. 
		
	\section{Examples}\label{sec:Examples}	
		We end this paper with examples of chirally and fully perfect colorings of tilings with a singular point obtained from ideal colorings of lattices associated with 
		regular tilings.
	
		\subsection*{$(4^4)$ Tiling}
			Let $\alpha=-5+5i$.
			Then $\stiling$ is a Class 2 tiling, and $\sg(\stiling)\cong D_5$ is generated by $g_1(z)=\exp(2\pi i/5)z$ and $g_2(z)=\bar{z}$.  
			The symmetries $g_1$ and $g_2$ of $\stiling$ correspond to the symmetries
			$f_{1}(z)=z+(-1+i)$ and  $f_{2}(z)=iz$ of $\tiling$, respectively.
			Ideal colorings of $P=\Z[i]$ compatible with $\cmap$ are induced by the ideals $(1)$, $(-1+i)$, $(1+2i)$, $(1-2i)$, $(1+3i)$, $(3+i)$, $(5)$ and $(-5+5i)$.
			The colorings of $\stiling$ that correspond to the colorings of $P$ induced by the ideals $(1)$, $(-1+i)$, $(5)$, and $(-5+5i)$ are fully perfect.
			The maximum color index is $50$.  Figure \ref{fig:(4^4)Example1a} shows a chirally perfect coloring of $\stiling$ obtained from the coloring of $P$ induced by 
			$(2+i)$.
								
			Figure \ref{fig:(4^4)Example2} shows a coloring of a Class 3 tiling $\stiling$, where $\alpha=4$, obtained from the coloring of $P$ induced by $(2)$. 
			Since $2$ is balanced, the coloring is fully perfect. 
						
			\begin{figure}[ht]
				\includegraphics[height=4.5cm]{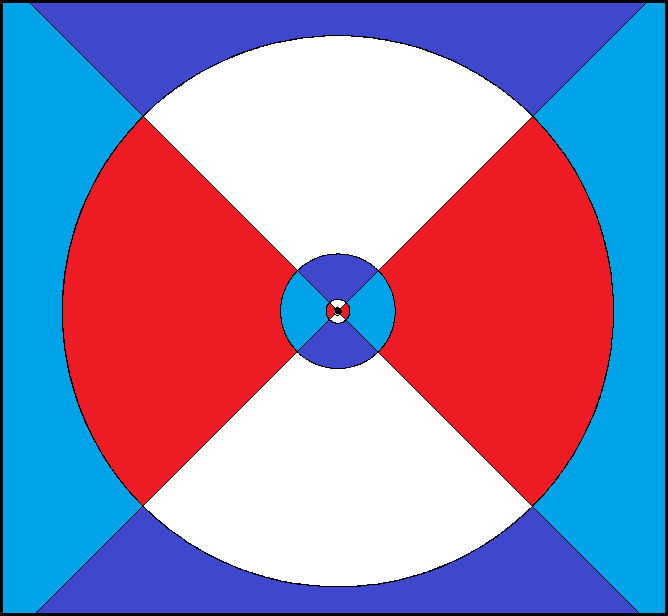}
				\caption{A fully perfect coloring of $\stiling$, with $\alpha=4$, obtained from the coloring of $P$ induced by $(2)$}
				\label{fig:(4^4)Example2}
			\end{figure}
	
		\subsection*{$(6^3)$ Tiling}
			If $\alpha=6$, then $\stiling$ is a Class 3 tiling.
			The generators of $\sg(\stiling)\cong D_6$ are $g_1(z)=\exp(\pi i/3) z$ and $g_2(z)=\bar{z}$, which correspond to the symmetries $f_1(z)=z+1$ and $f_2(z)=-\bar{z}$
			of $\tiling$, respectively.
			The ideals $(1)$, $(2+\omega)$, $(2)$, $(3)$, $(2+4\omega)$, and $(6)$ induce colorings of $P=\Z[\omega]$ that are compatible with $\cmap$.
			All resulting colorings of $\stiling$ are perfect, and the maximum color index is $36$.
			Figure \ref{fig:(6^3)Example1} shows the perfect colorings of $\tiling$ and $\stiling$ induced by $(2)$.
			
			\begin{figure}[ht]
				\begin{subfigure}[c]{0.3\linewidth}
					\includegraphics[height=4.5cm]{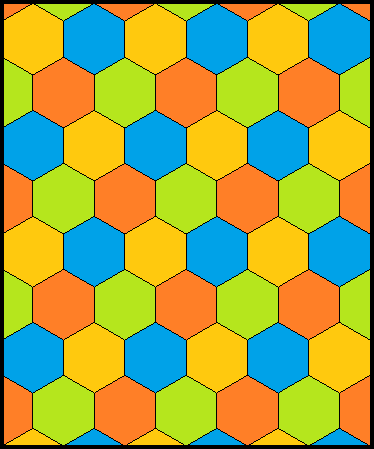}
				\end{subfigure}
				\begin{subfigure}[c]{0.5\linewidth}
					\includegraphics[height=4.5cm]{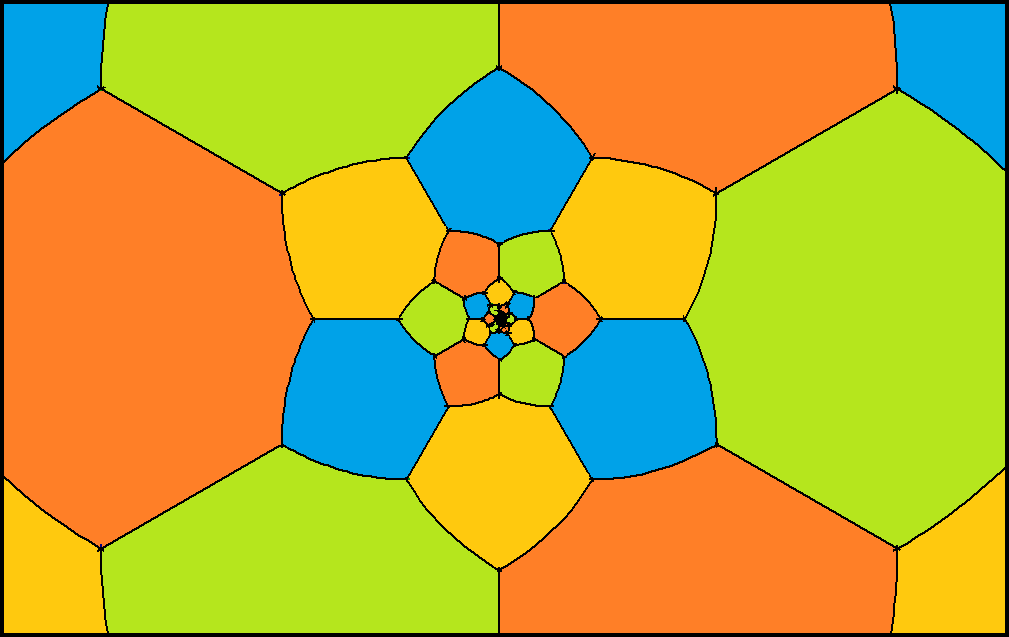}
				\end{subfigure}
				\caption{Perfect colorings of $\tiling$ and $\stiling$, with $\alpha=6$, induced by $(2)$}
				\label{fig:(6^3)Example1}
			\end{figure}			
			
			In Figure \ref{fig:(6^3)Example2}, we have a coloring of a Class 1 tiling $\stiling$ with $\alpha=10+2\omega$.  
			The coloring is perfect and is induced by the ideal $(1+3\omega)$. 
			
			\begin{figure}[ht]
				\includegraphics[height=4.5cm]{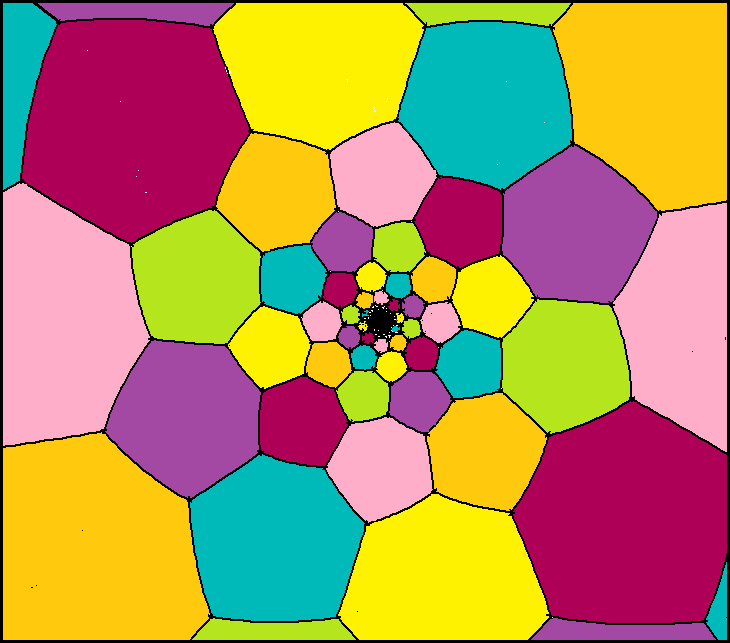}
				\caption{A perfect coloring of $\stiling$, with $\alpha =10+2\omega$, obtained from a coloring of $P$ induced by $(1+3\omega)$}
				\label{fig:(6^3)Example2}
			\end{figure}
	
		\subsection*{$(3^6)$ Tiling}	
			We obtain a Class 1 tiling $\stiling$ when $\alpha=2+10\omega=(4+6\omega)(2+\omega)$.  The symmetry group of $\stiling$ is isomorphic to $C_2$, with generator $g_1(z)=-z$
			that corresponds to the symmetry $f(z)=z+(1+5\omega)$ of $\tiling$.
			Colorings of $P=\Z[\omega]\setminus (2+\omega)$ compatible with $\cmap$ are induced by the ideals $(1)$, $(1+2\omega)$, $(2)$, $(2+3\omega)$, $(2+4\omega)$, $(1+5\omega)$, 
			$(4+6\omega)$, and $(2+10\omega)$.  The maximum color index is $56$.   
			Figure \ref{fig:(3^6)Example1} shows the chirally perfect coloring of $\mathscr{T}$ and the fully perfect coloring of $\stiling$ induced by $(2+3\omega)$. 
			
			\begin{figure}[ht]
				\begin{subfigure}[c]{0.3\linewidth}
					\includegraphics[height=4.5cm]{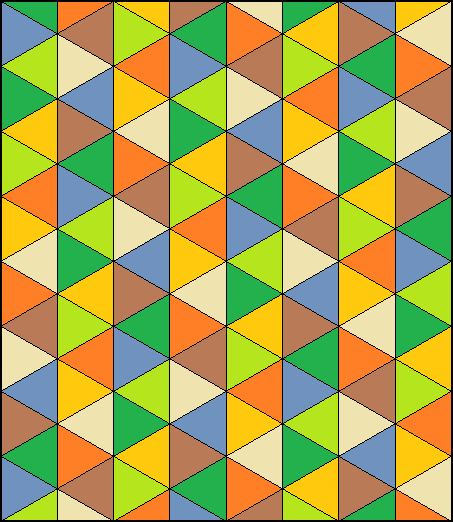}
				\end{subfigure}
				\begin{subfigure}[c]{0.6\linewidth}
					\includegraphics[height=4.5cm]{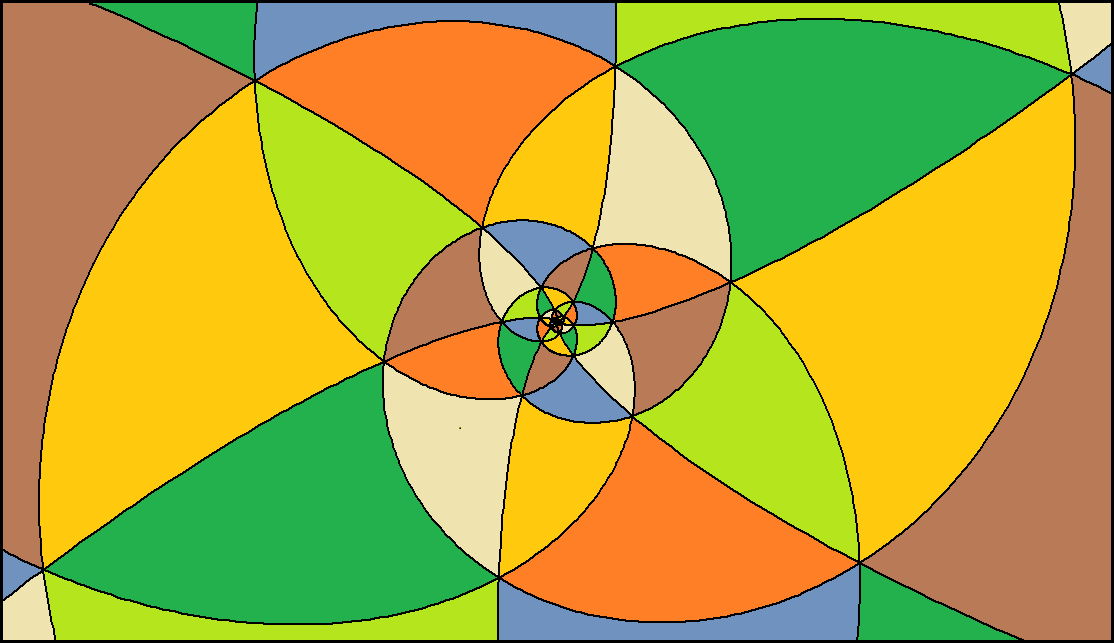}
				\end{subfigure}
				\caption{A perfect coloring of $\stiling$, with $\alpha =2+10\omega$, obtained from a coloring of $P$ induced by the ideal $(2+3\omega)$}
				\label{fig:(3^6)Example1}
			\end{figure}
			
			Figure \ref{fig:(3^6)Example2} shows the fully perfect coloring of a Class 2 tiling $\stiling$, where $\alpha =-5+5\omega=5\omega(2+\omega)$, obtained from a coloring of 
			$P$ induced by $(1+2\omega)$.
					
			\begin{figure}[ht]
				\includegraphics[height=4.5cm]{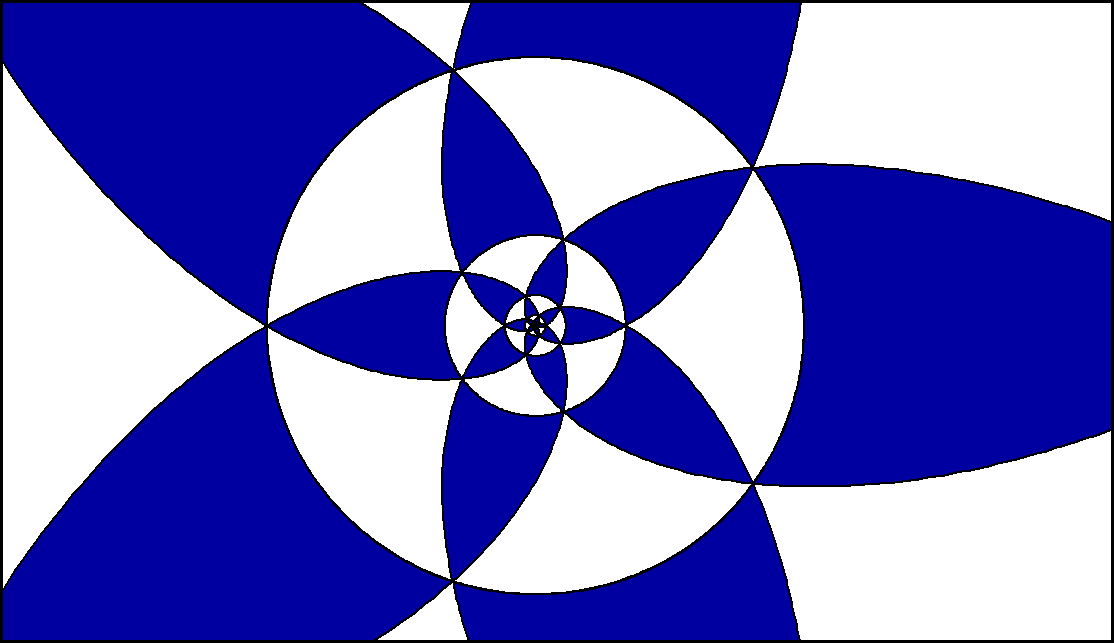}
				\caption{A fully perfect coloring of $\stiling$, with $\alpha =-5+5\omega$, obtained from a coloring of $P$ induced by $(1+2\omega)$}
				\label{fig:(3^6)Example2}
			\end{figure}


\begin{thebibliography}{99}	
		\bibitem{BaakeGrimm} 
			M.~Baake and U.~Grimm, \textit{Aperiodic order.~Vol.~1.~A mathematical invitation}, Encyclopedia of Mathematics and its Applications, vol.~149, Cambridge University Press, 
			Cambridge, 2013.		
		\bibitem{Bieb}
		 	L.~Bieberbach, \textit{Conformal mapping}, Chelsea Publishing Co., New York, 1953.
		\bibitem{Bree}
			M.~Breen, \textit{Some tilings of the plane whose singular points form a perfect set}, Amer.~Math.~Soc.~\textbf{89} no.~3 (1983) 477--479.
		\bibitem{Buga}
			E.P.~Bugarin, M.L.A.N.~de las Pe\~{n}as, and D.~ Frettl\"{o}h, \textit{Perfect colourings of cyclotomic integers}, Geom.~Dedicata \textbf{162} (2013) 271--282.
		\bibitem{ConwaySloane}
			J.H.~Conway and N.J.A.~Sloane, \textit{Sphere packings, lattices and groups}, 3rd ed., Fundamental Principles of Mathematical Sciences, vol.~290, Springer, New York, 1999.
		\bibitem{Pena}
			M.L.A.N. de las Pe\~{n}as and R.P.~Felix, \textit{Color groups associated with square and hexagonal lattices.} Z.~Kristallogr.~\textbf{222} (2007) 505--512.
		\bibitem{Grun}
		 	B.~Gr\"{u}nbaum and G.C.~Shephard, \textit{Tilings and patterns}, W.H.~Freeman and Company, New York, 1987.
		\bibitem{Luck}
			R.~L\"{u}ck, \textit{Color groups in tilings with singularities}, J.~Phys.: Conf.~Ser.~\textbf{226} (2010) 012027. 
		\bibitem{Mood}
		 	R.V.~Moody and J.~Patera, \textit{Colourings of quasicrystals}, Can.~J.~Phys.~\textbf{72}, 442 (1994) 442--452.
		\bibitem{Niel}
			M.J.~Nielsen, \textit{Singular points	of a star-finite tiling}, Geom.~Dedicata \textbf{33} (1990) 99--109.
		\bibitem{Rigb}
			J.F.~Rigby, \emph{Precise colourings of regular triangular tilings}, Math.~Intelligencer \textbf{20} no.~1 (1998) 4--11.
		\bibitem{Sush}
			T.~Sushida, A.~Hizume, and Y.~Yamagishi, \textit{Triangular spiral tilings}, J.~Phys.~A~\textbf{45} no.~23 (2012) 235203.		
		\bibitem{Wash}
			L.C.~Washington, \textit{Introduction to cyclotomic fields}, 2nd ed., Graduate Texts in Mathematics, vol.~83, Springer, New York, 1996.
	\end{thebibliography}
\end{document}